\documentclass[12pt]{amsart}
\usepackage{amssymb,amsmath}
\def\L {{\mathcal L }}
\def\E {{\mathcal E}}
\def\H {{\mathcal H}}
\def\Q {{\mathcal Q}}

\def\M {{\mathcal M}}

\def\A {{\mathfrak A}}
\def\B {{\mathcal B}}

\def\DD {\mathbb{D}}
\def\R {\mathbb{R}}
\def\N {\mathbb{N}}

\def\D {{\mathfrak D}}
\def\I {{\mathfrak I}}

\def\HH{{\rm H}}

\def\eps{\varepsilon}
\def\e{{\rm e}}

\def\d{{\rm d}}
\def\ddt{\frac{\d}{\d t}}
\def\ddx{\frac{\d}{\d x}}

\def \l {\langle}
\def \r {\rangle}

\def \norm {\boldsymbol{|}}

\def \and {{\qquad\text{and}\qquad}}

\newtheorem{proposition}{Proposition}[section]
\newtheorem{theorem}[proposition]{Theorem}

\newtheorem{lemma}[proposition]{Lemma}
\theoremstyle{definition}
\newtheorem{definition}[proposition]{Definition}
\newtheorem{remark}[proposition]{Remark}
\numberwithin{equation}{section}
\def \au {\rm}
\def \ti {\it}
\def \jou {\rm}
\def \bk {\it}
\def \no#1#2#3 {{\bf #1} (#3), #2.}
\def \eds#1#2#3 {#1, #2, #3.}
\title[Benjamin-Bona-Mahony equation with memory]
{Global attractors for the\\
Benjamin-Bona-Mahony equation with memory}

\author[F. Dell'Oro, O. Goubet, Y. Mammeri and V. Pata]
{Filippo Dell'Oro, Olivier Goubet, Youcef Mammeri and Vittorino Pata}

\address{Politecnico di Milano - Dipartimento di Matematica
\newline\indent
Via Bonardi 9, 20133 Milano, Italy}
\email{filippo.delloro@polimi.it {\rm (F. Dell'Oro)}}
\email{vittorino.pata@polimi.it {\rm (V. Pata)}}

\address{Laboratoire Ami\'enois de Math\'ematique Fondamentale et Appliqu\'ee
\newline\indent
CNRS UMR 7352, Universit\'e de Picardie Jules Verne, 80069 Amiens, France}
\email{olivier.goubet@u-picardie.fr {\rm (O. Goubet)}}
\email{youcef.mammeri@u-picardie.fr {\rm (Y. Mammeri)}}

\subjclass[2000]{35B41, 35F25, 45K05}
\keywords{Benjamin-Bona-Mahony equation, dissipative memory, global attractors}
\begin{document}

\begin{abstract}
We consider the nonlinear integrodifferential Benjamin-Bona-Mahony equation
$$
u_t - u_{txx} + u_x - \int_0^\infty g(s) u_{xx}(t-s) \d s + u u_x = f
$$
where the dissipation is entirely contributed by the memory term.
Under a suitable smallness assumption on the external force $f$,
we show that the
related solution semigroup possesses the global attractor
in the natural weak energy space. The result is obtained by means of a nonstandard
approach based on the construction of a suitable family of attractors on certain invariant
sets of the phase space.
\end{abstract}

\maketitle

%%%%%%%%%%%%%%%%%%%%%%%%%%%%%%%%%%%%%%%%%%%%%%%%%

%%%%%%%%%%%%%%%%%%%%%%%%%%%%%%%%%%%%%%%%%%%%%%%%%
\section{Introduction}

\noindent
In 1972, Benjamin, Bona and Mahony~\cite{BBM} introduced the nonlinear equation (from now on called BBM)
\begin{equation}
\label{BBMeq}
u_t - u_{txx} + u_x + u u_x =f,
\end{equation}
in the unknown variable $u=u(x,t):\I\times \R^+ \to \R$, where
$$\I=(a,b)$$
is a (bounded) interval of the real line.
The equation is
supplemented with the Dirichlet boundary condition
\begin{equation}
\label{BC}
u(a,t) = u(b,t) = 0.
\end{equation} From the physical viewpoint, $u$ represents
the one-directional amplitude of long waves
in shallow water, whereas $f\in L^2(\I)$ is a time-independent external force.
It is interesting to note that in the homogeneous case, i.e.\ when $f=0$, the natural energy
$$\E(t)=\|u(t)\|^2+\|u_x(t)\|^2$$
is a conserved quantity, where $\|\cdot\|$ denotes the norm in $L^2(\I)$.
This can be easily verified multiplying the equation by $u$, taking into account the
boundary condition \eqref{BC}.

As a matter of fact, \eqref{BBMeq} is obtained from the Korteweg-de Vries equation \cite{KDV}
\begin{equation}
\label{KdV}
u_t + u_{xxx} + u_x + u u_x =f,
\end{equation}
merely by replacing the term $u_{xxx}$ by $-u_{txx}$. To some extent,
equation~\eqref{BBMeq} can be seen as a regularized version
of \eqref{KdV}.
In the dissipative case, that is,
in the modeling of long gravity waves where the viscosity $\nu>0$ of the fluid is not neglected
(see e.g.\ \cite{BYATT,Dut,KAKUTANI}), the BBM equation displays the extra term
$- \nu u_{xx}$.
Accordingly, taking $\nu=1$ for simplicity, the dissipative version of~\eqref{BBMeq} reads
\begin{equation}
\label{dBBM}
u_t  - u_{txx}  + u_x - u_{xx} +  u u_x = f.
\end{equation}
The longterm dynamics of \eqref{dBBM} has been the object of several investigations. In particular,
the existence of a finite-dimensional
global attractor for the related solution semigroup has been proved by Wang and Yang \cite{WAN, WAYA}.
Other results can be found in \cite{ABS,BD,GUO,LARKIN,STA,UAUA} and references therein.

In the recent work \cite{DMP}, some of the authors of the present paper proposed a memory relaxation
of the dissipative BBM equation~\eqref{dBBM}. More precisely, they considered the integrodifferential problem
in the variable $u=u(x,t):\I\times \R \to \R$
\begin{equation}
\label{BBM}
u_t - u_{txx} + u_x - \int_0^\infty g(s) u_{xx}(t-s) \d s + u u_x = f,
\end{equation}
subject to the Dirichlet boundary condition~\eqref{BC}.
Here,
$g:[0,\infty)\to[0,\infty)$ is a convex summable
kernel of unitary total mass, and
the function $u$ is supposed to be known for all $t\leq0$, where it plays the role of an initial datum.
The delay induced by the presence of the memory destroys the parabolic character of the
BBM equation~\eqref{dBBM},
providing a more realistic description of the phenomenon. In particular, it prevents
the unphysical feature of the instantaneous regularization of initial data
(see e.g.\ \cite{Fichera,GP} where the same issue is discussed in a different context).
It is also worth noting that the dissipative BBM model~\eqref{dBBM}
is formally recovered from \eqref{BBM} in the limiting situation
where $g$ collapses into the Dirac mass at zero.

Mathematically speaking, there are remarkable differences between
the dissipative equations \eqref{dBBM} and \eqref{BBM}.
In particular, in the homogeneous case $f=0$, the exponential stability of \eqref{dBBM}
is quite easy to prove:
the natural multiplication by $u$ immediately gives the differential identity
$$\ddt \E+2\|u_x\|^2=0,$$
and by means of the Poincar\'e inequality one readily obtains
$$\ddt \E+\varkappa \E\leq0,\quad\varkappa>0.$$
On the contrary, being entirely contributed by the memory term, the dissipation mechanics
of \eqref{BBM} is much weaker (and nonlocal), and
the basic energy identity alone is not sufficient to
provide the exponential decay of the solutions. Indeed, since the instantaneous damping no longer appears,
one needs
to introduce an auxiliary energy-like functional in order to reconstruct the missing term $\|u_x\|^2$.
On the other hand, when dealing
with such an auxiliary functional, the treatment  of the nonlinear
term (that cancels out when performing the basic energy estimate) becomes quite delicate.
Nonetheless it is still true that the (nonlinear) semigroup
generated by the homogeneous version of \eqref{BBM} is exponentially stable.
This is the content of the paper~\cite{DMP}, whose key idea was to exploit
in a crucial way the gradient-system structure of the problem,
together with a recursion argument.
Summarizing, similarly to what happens in the Navier-Stokes system (see e.g.\ \cite{temNS})
the asymptotic dynamics of \eqref{BBM} with $f=0$ is trivial, and all the complexity arises in presence
of the external force.

The purpose of this work is exactly the longterm analysis of the solutions to \eqref{BBM}
with a nonzero term $f$. The introduction of the external force
renders the picture much more complicated from the very beginning, since
the gradient-system structure is completely lost. In particular, the techniques of~\cite{DMP}
no longer apply.
Our main result is the existence of the (regular) global attractor for the solution
semigroup $S(t)$ generated by \eqref{BBM}, rewritten as a dynamical system in the so-called history
space framework of Dafermos~\cite{DAF}.
This can be done under a suitable smallness
assumption on $f$. The lack of the gradient-system structure,
combined with the extremely weak dissipation mechanism provided by the memory,
makes the problem highly nontrivial.
Our strategy here is to follow
a nonstandard approach, based on the construction of a family of attractors $\A_\eps$ on certain invariant
sets $\DD_\eps$ of the phase space. When $\eps\to 0$, the sets $\DD_\eps$ turn out to fill the space,
and the attractors $\A_\eps$ are shown to coincide. This allows us to conclude.

\subsection*{Plan of the paper}
In the next Section \ref{S2} we introduce the functional setting and the notation, while in the subsequent
Section \ref{S3} we establish the existence of the solution semigroup $S(t)$.
In Section~\ref{S4} we state the main result about the global attractor,
whose proof is carried out in the remaining Sections \ref{S5}-\ref{S11}.
In particular, Section \ref{S5} deals with two ODE lemmas needed in the course of the investigation, while
Sections \ref{S6}-\ref{S8} are devoted to the construction of the family of
invariant sets $\DD_\eps $. The restriction of $S(t)$ on $\DD_\eps$ is then shown to possess the global
attractor (see Section \ref{S10}).
In the final Section \ref{S11}, making use of a technical lemma proved in Section \ref{S9}, we complete the proof of the main result.
%%%%%%%%%%%%%%%%%%%%%%%%%%%%%%%%%%%%%%%%%%%%%%%%%

%%%%%%%%%%%%%%%%%%%%%%%%%%%%%%%%%%%%%%%%%%%%%%%%%
\section{Functional Setting and Notation}
\label{S2}

\subsection{Geometric spaces}
Calling $\HH=L^2(\I)$ with
inner product $\l\cdot,\cdot\r$ and norm $\|\cdot\|$,
we introduce the strictly positive selfadjoint Dirichlet
operator
$$A = -\partial_{xx} \qquad\text{with domain}\qquad \D(A) = H^2(\I)\cap H_0^1(\I)\Subset \HH,$$
$H^2(\I)$ and $H_0^1(\I)$ being the usual Sobolev spaces on the interval $\I$.
Then, for $r\in\R$, we define the compactly nested
family of Hilbert spaces
($r$ will be always omitted whenever zero)
$$
\HH^r=\D(A^{\frac{r}2}),\qquad\l u,v\r_r=\l A^{\frac{r}2}u,A^{\frac{r}2}v\r,\qquad
\|u\|_r=\|A^{\frac{r}2}u\|.
$$
The symbol $\l\cdot,\cdot\r$
also stands for duality product between $\HH^r$ and its dual space $\HH^{-r}$. In particular,
$$
\HH^2 = H^2(\I)\cap H_0^1(\I)\Subset \HH^1 = H_0^1(\I)\Subset \HH = L^2(\I),
$$
and we have the
Poincar\'e inequalities
$$
\sqrt{\lambda_1}\| u \|_r \leq \|u\|_{r+1},\quad\,\, \forall u \in \HH^{r+1},
$$
where $\lambda_1>0$ is the first eigenvalue of $A$.
These inequalities, as well as the H\"older and the Young inequalities, will be
used several times in what follows, often without explicit mention.
In order to simplify the calculations, we also consider the strictly positive selfadjoint operator
$$
B = I + A\qquad\text{with domain}\qquad\D(B)=\D(A).
$$
Such an operator $B$ commutes with $A$, and the bilinear form
$$
( u, v )_r =\l B^{\frac12}u,B^{\frac12}v \r_{r-1} = \l A^{\frac{r-1}{2}}B^{\frac12}u, A^{\frac{r-1}{2}}B^{\frac12}v\r
$$
defines an equivalent inner product on $\HH^r$, with induced norm
$$
\norm u \norm_r^2 = \|u\|_{r-1}^2 + \|u\|_r^2.
$$
In particular, setting
$$
\omega = \sqrt{\frac{1+\lambda_1}{\lambda_1}}>1,
$$
and exploiting the Poincar\'e inequalities, we get
\begin{equation}
\label{EQUIV}
\|u\|_{r}\leq \norm u \norm_r \leq \omega\|u\|_{r},\quad\,\, \forall u \in \HH^{r}.
\end{equation}

\subsection{Assumptions on the memory kernel}
The function $g$ is supposed to have the explicit form
$$
g(s)=\displaystyle \int_s^\infty \mu(y)\d y,
$$
where the so-called memory kernel $\mu\not\equiv 0$
is a nonnegative, nonincreasing and absolutely continuous function on $\R^+=(0,\infty)$. In particular $\mu$
is summable on $\R^+$ with
$$
g(0)=\int_0^\infty \mu(s)\d s\doteq\kappa>0,
$$
while the requirement that $g$ has total mass $1$ translates into
$$
\int_0^\infty s\mu(s)\,\d s=1.
$$
In addition, let the following structural conditions hold.

\smallskip
\begin{itemize}
\item[{\bf (M1)}] $\mu$ is bounded about zero, namely,
$$
\mu(0) = \lim_{s\to 0^+} \mu(s) <\infty.
$$
\item[{\bf (M2)}] $\mu$ satisfies for some $\delta>0$ and almost every $s>0$
the {\it Dafermos condition}
$$
\mu'(s) + \delta \mu(s) \leq 0.
$$
\end{itemize}

\subsection{Memory spaces}
We now consider the {\it memory spaces} (again $r$ will be omitted whenever zero)
$$
\M^r = L^2_\mu(\R^+; \HH^{r+1})
$$
of square summable $\HH^{r+1}$-valued functions on $\R^+$ with respect to the measure $\mu(s)\d s$,
endowed with the weighted inner product
$$\langle\eta,\xi\rangle_{\M^r}=\int_0^\infty \mu(s)\l\eta(s),\xi(s)\r_{r+1}\,\d s,
$$
with induced norm
$$\|\eta\|_{\M^r}=\bigg(\int_0^\infty \mu(s)\|\eta(s)\|_{r+1}^2\,\d s\bigg)^\frac12.
$$
We will also work with the equivalent inner product
$$
(\eta,\xi)_{\M^r}=\int_0^\infty \mu(s)(\eta(s),\xi(s))_{r+1}\,\d s,
$$
with induced norm $\norm\cdot\norm_{\M^r}$.
The infinitesimal generator of the right-translation semigroup on $\M$ is
the linear operator
$$T\eta=-\partial_s\eta\qquad\text{with domain}\qquad\D(T)=\big\{\eta\in{\M}:\,\partial_s\eta\in\M \,\,\,\text{and}\,\,\,
\lim_{s\to 0}\|\eta(s)\|_1=0\big\},$$
where $\partial_s$ stands for weak derivative with respect to the internal variable $s\in\R^+$.
For every $\eta\in\D(T)$, we introduce the nonnegative functional
$$
\Gamma[\eta] = -\int_0^\infty \mu'(s) \|\eta(s)\|_{1}^2.
$$
Exploiting the Dafermos condition (M2), it is apparent to see that
\begin{equation}
\label{gamm}
\Gamma[\eta] \geq \delta \|\eta\|_\M^2.
\end{equation}
Moreover,
an integration by parts together with a limiting argument yield (see e.g.\ \cite{CheP,Terreni})
\begin{equation}
\label{FUNCGAMMA}
\Gamma[\eta] = -2 \l T \eta ,\eta \r_{\M}.
\end{equation}

\subsection{Extended memory spaces}
Finally, we define the {\it extended memory spaces}
$$
\H^r = \HH^{r+1} \times \M^{r}
$$
endowed with the product norm
$$
\|(u,\eta)\|_{\H^r}^2 = \norm u \norm_{r+1}^2 +  \|\eta\|_{\M^{r}}^2.
$$
The phase space of our problem will be
$$
\H = \HH^{1} \times \M.
$$
%%%%%%%%%%%%%%%%%%%%%%%%%%%%%%%%%%%%%%%%%%%%%%%%%

%%%%%%%%%%%%%%%%%%%%%%%%%%%%%%%%%%%%%%%%%%%%%%%%%
\section{The Solution Semigroup}
\label{S3}

\noindent
We translate equation \eqref{BBM} in the
history space framework of Dafermos \cite{DAF}. To this end,
defining the additional variable
$$
\eta^t(x,s)=\int_{0}^s u(x,t-y)\,\d y,
$$
accounting for the integrated past history of $u$, we rewrite \eqref{BBM} subject
to the boundary condition~\eqref{BC} as
\begin{equation}
\label{base}
\begin{cases}
\displaystyle
B u_t + u_x + \int_0^\infty \mu(s) A \eta(s) \d s + u u_x=f,\\\noalign{\vskip0.5mm}
\eta_t = T\eta + u.
\end{cases}
\end{equation}
By means of a standard Galerkin approximation scheme, or
using the approach recently devised in \cite{CDGP}, system \eqref{base} above is shown to generate a strongly continuous
semigroup
$$
S(t):\H \to \H.
$$
Hence, for every initial datum $z\in\H$, the unique solution at time $t>0$
is given by
$$
S(t)z = (u(t),\eta^t),
$$
whose related (twice the) energy reads
$$
\E(t) = \|S(t)z\|^2_\H = \norm u(t)\norm_1^2 + \|\eta^t\|_\M^2.
$$
In addition, for every $R\geq 0$
there exists an increasing positive function $\Q_R(\cdot)$
such that the continuous dependence estimate
$$\|S(t)z_{1}-S(t)z_{2}\|_\H\leq \Q_R(t)\|z_{1}-z_{2}\|_\H$$
holds for all initial data $z_i$ with $\|z_i\|_\H\leq R$.

\begin{proposition}
\label{enin}
For all sufficiently regular initial data, we have the energy identity
$$
\ddt \E +\Gamma[\eta] =2 \l f ,u\r,
$$
where $\Gamma[\eta]$ is given by \eqref{FUNCGAMMA}.
\end{proposition}

\begin{proof}
We multiply the first equation of \eqref{base} by $2u$ in $\HH$ and the second one
by $2\eta$ in $\M$. Taking the sum and making use of \eqref{FUNCGAMMA}, we get
$$
\ddt \E +\Gamma[\eta] + 2 \l u_x , u\r + 2 \l u u_x,u\r =  2 \l f ,u\r.
$$
Exploiting the Dirichlet boundary condition \eqref{BC},
$$
2 \l u_x , u\r + 2 \l u u_x,u\r = \int_a^b \ddx u^2(x)\d x + \frac{2}{3} \int_a^b \ddx u^3(x)\d x = 0,
$$
and the conclusion follows.
\end{proof}
%%%%%%%%%%%%%%%%%%%%%%%%%%%%%%%%%%%%%%%%%%%%%%%%%

%%%%%%%%%%%%%%%%%%%%%%%%%%%%%%%%%%%%%%%%%%%%%%%%%
\section{Statement of the Main Result}
\label{S4}

\noindent
The most relevant object in the longterm analysis
of a semigroup is the {\it global attractor} (see e.g. \cite{BV,HAL,TEM}). Let us recall the definition.

\begin{definition}
The global attractor of the semigroup $S(t)$ is the unique compact set $\A\subset\H$
which is at the same time

\begin{itemize}
\item[(i)] {\it fully invariant:}\ $\,S(t)\A = \A$ for every $t\geq0$; and

\smallskip
\item[(ii)] {\it attracting:}\ $\,\boldsymbol{\delta}(S(t)\B,\A)\to0$ as $t\to\infty$ for any bounded set
$\B\subset\H$.

\end{itemize}
\end{definition}

In the usual notation,
$$\boldsymbol{\delta}(\B_1,\B_2) = \sup_{z_1\in\B_1} \inf_{z_2\in\B_2} \|z_1-z_2\|_\H
$$
is the Hausdorff semidistance between two (nonempty) sets $\B_1,\B_2\subset\H$.

\begin{remark}
The notion of attractor can also be given for the {\it restriction} of $S(t)$
on any closed subset $\DD\subset \H$ (hence, a complete metric space with the distance inherited by $\H$), provided that $\DD$ is
{\it invariant} for the semigroup, i.e.\
$$S(t)\DD\subset\DD,\quad\forall t\geq 0.$$
In this case, the definition above makes sense simply by replacing the whole space $\H$ with $\DD$.
In particular, if the restriction of $S(t)$ on $\DD$ possesses the attractor $\A'$,
then $\A'$ is the largest fully invariant bounded subset of $\DD$,
namely, for every fully invariant bounded set $\B\subset\DD$ the inclusion
$\B \subset\A'$ holds. This fact will be heavily used in the sequel.
\end{remark}

The main result of the paper establishes the existence of the global attractor
of the semigroup $S(t)$ on $\H$
under a suitable smallness assumption on the
primitive
\begin{equation}
\label{EFFEBIG}
F(x) = \int_a^x f(y) \d y, \quad x\in\I=(a,b),
\end{equation}
of the external force $f$.

\begin{theorem}
\label{attractor}
There exists a structural constant $\mathfrak c>0$ such that if
$$\|F\|<{\mathfrak c}$$
then the semigroup $S(t):\H\to\H$ possesses the global attractor $\A$.
Moreover, $\A$ is a
bounded subset of $\H^1$.
\end{theorem}

The constant $\mathfrak c$ above is {\it independent} of $f$, and can be explicitly calculated in terms of the
other structural quantities of the problem.

\smallskip
The remaining of the paper is devoted to the proof of Theorem \ref{attractor}.
In what follows, we will always assume $f\not\equiv 0$. This in particular implies $\|F\|>0$.
Indeed, according to \cite{DMP}, when $f\equiv 0$
exponential stability occurs.
%%%%%%%%%%%%%%%%%%%%%%%%%%%%%%%%%%%%%%%%%%%%%%%%%

%%%%%%%%%%%%%%%%%%%%%%%%%%%%%%%%%%%%%%%%%%%%%%%%%
\section{Two Lemmas from ODEs}
\label{S5}

\noindent
We begin with two technical ODE results
needed in the course of the  investigation.
Let $\L\in\mathcal{C}^1([0,\infty))$ be a fixed function satisfying for every $t\geq0$
the differential inequality
\begin{equation}
\label{DI}
\L'(t) + 2b \L(t) \leq c +a \L^2(t),
\end{equation}
for some $a,b,c>0$ subject to the structural constraint
$$
\frac{b}{\sqrt{ac}}\doteq \varrho >1.
$$
In particular, this implies that\footnote{Note that $\lambda_{\pm}$ are the roots of the equation
$ax^2 - 2bx+c=0$.}
$$
\lambda_{-}\doteq\sqrt{\frac{c}{a}}(\varrho-\sqrt{\varrho^2-1}) < \sqrt{\frac{c}{a}}(\varrho+\sqrt{\varrho^2-1})\doteq\lambda_+.
$$

\begin{lemma}
\label{lemmaODE1}
Let $\lambda\in (\lambda_{-},\lambda_+)$ be arbitrarily chosen.
Then, the following implication holds:
$$
\L(0)\leq \lambda  \quad \Rightarrow \quad \sup_{t\geq0}\L(t) \leq \lambda.
$$
\end{lemma}

\begin{proof}
Since $\L$ is continuous, let us define
$$
t_* = \max \big\{ \tau\geq0 :\, \L(s) \leq \lambda, \, \forall s \in [0,\tau]\big\}.
$$
Our aim is showing that $t_* = \infty$. If not,
$\L(t_*)=\lambda$ and by \eqref{DI}
$$
\L'(t_*) \leq a (\lambda - \lambda_{-})(\lambda - \lambda_{+})<0.
$$
As a consequence, the function $\L$ is decreasing in a right neighborhood of $t_*$,
contradicting the maximality of $t_*$.
\end{proof}

\begin{lemma}
\label{lemmaODE2}
There exists a time $t_\varrho>0$ depending only
on $\varrho$ such that the following implication holds:
$$ \L(0)\leq \sqrt{\frac{c}{a}} (2\varrho-1)\quad \Rightarrow \quad
\sup_{t\geq t_\varrho/\sqrt{ac}}  \L(t) \leq   \sqrt{\frac{c}{a}}\left(\frac{1}{2\varrho-1}\right).
$$
\end{lemma}

\begin{proof}
Being $\varrho>1$, it is immediate to verify that
$$
\sqrt{\frac{c}{a}} (2\varrho-1) \in (\lambda_{-},\lambda_+).
$$
Therefore, applying Lemma \ref{lemmaODE1} with $\lambda = \sqrt{\tfrac{c}{a}} (2\varrho-1) $,
$$
\sup_{t\geq0} \L(t) \leq \sqrt{\frac{c}{a}} (2\varrho-1).
$$
At this point, in order to simplify the computations, we introduce the auxiliary function
$$y(t)=\sqrt\frac{a}{c}\L\bigg(\frac{t}{\sqrt{ac}}\bigg),$$
along with the number
$$r=\varrho+\sqrt{\varrho^2-1} > 2\varrho-1>1.$$
In particular, the inequality above provides the control
\begin{equation}
\label{CONT}
\sup_{t\geq0} y(t) \leq 2\varrho-1.
\end{equation}
Exploiting now \eqref{DI}, by means of direct calculations we see that
$$
\ddt \Big(y(t)-\frac1r\Big) = y' (t)\leq (y(t)-r) \Big(y(t)-\frac 1 r\Big).
$$
Hence, appealing to the Gronwall lemma and \eqref{CONT}, we get the estimate
\begin{align*}
y(t)-\frac 1 r&\leq \Big(y(0)-\frac1r\Big)\e^{\int_0^t (y(s)- r)\d s}\\
&\leq \Big(2\varrho-1-\frac 1 r\Big)\e^{\int_0^t (y(s)- r)\d s}\\
&\leq \Big(2\varrho-1-\frac 1 r\Big)\e^{(2\varrho-1-r)t},
\end{align*}
valid for every $t\geq0$. Note that $2\varrho-1-\frac 1 r>0$. Calling
$$
t_\varrho = \frac{1}{(2\varrho-1-r)}\log \bigg( \frac{r+1-2\varrho}{4 \varrho r(\varrho-1)+r+1-2\varrho} \bigg)>0,
$$
which solves the equation
$$\Big(2\varrho-1-\frac 1 r\Big)\e^{(2\varrho-1-r)t_\varrho}=\frac{1}{2\varrho-1}-\frac 1 r,$$
we conclude that
$$y(t)\leq \frac1{2\varrho-1},\quad\forall t\geq t_\varrho.$$
Returning to the original $\L$, the proof is finished.
\end{proof}
%%%%%%%%%%%%%%%%%%%%%%%%%%%%%%%%%%%%%%%%%%%%%%%%%

%%%%%%%%%%%%%%%%%%%%%%%%%%%%%%%%%%%%%%%%%%%%%%%%%
\section{A Family of Geometric Functionals}
\label{S6}

\noindent
For any given $z=(u,\eta)\in\H$ and $\eps>0$, we introduce the geometric functional
$$
\Lambda_\eps (z) = \|z\|_{\H}^2
+ \frac{2}{\kappa} \int_0^\infty \mu(s) \l F , \eta_x(s)\r \d s
+ \frac{2}{\kappa} \|F\|^2  -\frac{\eps}{\sqrt\kappa}  \int_0^\infty \mu(s) ( u ,  \eta(s))_1 \,\d s,
$$
with $F$ as in \eqref{EFFEBIG}.
Exploiting the H\"older and Young inequalities, together with \eqref{EQUIV},
it is readily seen that
\begin{equation}
\label{geoen}
\frac{(1-\eps\omega)}{2}\|z\|_{\H}^2 \leq
\Lambda_\eps(z) \leq \frac{(3+\eps\omega)}{2}\|z\|_{\H}^2 + \frac4\kappa\|F\|^2.
\end{equation}
We now show that these functionals are all equivalent, at least if $\eps$ is sufficiently small.

\begin{lemma}
For every $0<\alpha<\eps< \tfrac{1}{2\omega}$, we have
\begin{equation}
\label{calculus1}
\Lambda_\alpha(z)\leq \frac{\Lambda_\eps(z)}{1-\omega\eps} \leq \frac{\Lambda_\alpha(z)}{1-2\omega\eps}.
\end{equation}
\end{lemma}

\begin{proof}
Making use of \eqref{EQUIV}, it is immediate to check that
$$
|\Lambda_\eps(z)-\Lambda_\alpha(z)| \leq (\eps-\alpha)\omega \norm u \norm_1 \|\eta\|_\M
\leq \frac{\eps\omega}{2}\|z\|_\H^2.
$$
Appealing to the first inequality in \eqref{geoen} we get
\begin{align*}
\Lambda_\alpha(z) \leq \Lambda_\eps(z)  + \frac{\eps\omega}{2}\|z\|_\H^2
\leq  \frac{\Lambda_\eps(z)}{1-\omega\eps}.
\end{align*}
By the same token,
\begin{align*}
\Lambda_\eps(z) \leq \Lambda_\alpha(z)  + \frac{\eps\omega}{2}\|z\|_\H^2
\leq  \Lambda_\alpha(z) + \frac{\eps\omega}{1-\eps\omega}  \Lambda_\eps(z).
\end{align*}
Hence
$$
\frac{1-2\eps\omega}{1-\eps\omega}\Lambda_\eps(z) \leq \Lambda_\alpha(z),
$$
which completes the proof.
\end{proof}
%%%%%%%%%%%%%%%%%%%%%%%%%%%%%%%%%%%%%%%%%%%%%%%%%

%%%%%%%%%%%%%%%%%%%%%%%%%%%%%%%%%%%%%%%%%%%%%%%%%
\section{A Family of Energy Inequalities}
\label{S7}

\noindent
Throughout the paper we will perform several formal estimates, which are fully justified
within a proper approximation scheme.
In order to study the longterm behavior of the semigroup $S(t)$, we need to
derive a suitable family of differential
inequalities for the energy-like functional
\begin{align*}
\L_\eps(t) &= \Lambda_\eps (S(t)z)\\
&=\E(t) + \frac{2}{\kappa} \int_0^\infty \mu(s) \l F , \eta_x^t(s)\r \d s
+ \frac{2}{\kappa} \|F\|^2 -\frac{\eps}{\sqrt\kappa} \int_0^\infty \mu(s) ( u(t) ,  \eta^t(s))_1 \,\d s.
\end{align*}
Observe that, in the light of \eqref{geoen},  the controls
\begin{equation}
\label{finalcont}
\frac14 \E(t)\leq \L_\eps(t) \leq 2 \E(t) + \frac{4}{\kappa} \|F\|^2
\end{equation}
hold for every $\eps\in(0,\tfrac{1}{2\omega})$
and every $t\geq0$.

\begin{lemma}
\label{fi}
There exist
$$c_1,c_2,c_3>0\and 0<\eps_0<\frac{1}{2\omega},$$
all
independent of $f$, and depending only on the other structural quantities of the problem,
such that the differential inequality
$$
\ddt \L_\eps(t)  + \eps c_1 \L_\eps(t) \leq c_2\|F\|^2 + c_3 \eps^2 \L_\eps^2(t)
$$
is satified for every $t\geq0$ and every $\eps\in (0,\eps_0]$.
\end{lemma}

\begin{proof}
Along the proof, $c\geq0$ will denote a {\it generic} constant independent of $f$ and the initial data.
An integration by parts provides the equality
$$
2 \l f ,u\r = - 2 \l F , u_x\r,
$$
where the boundary terms vanish due to \eqref{BC}. Accordingly, the energy identity of
Proposition \ref{enin}
takes the form
$$
\ddt \E +\Gamma[\eta] = - 2 \l F , u_x\r.
$$
Then, we compute the time
derivative of the functional $\L_\eps$ as
\begin{align*}
\ddt \L_\eps &= \ddt \E + \frac{2}{\kappa} \int_0^\infty \mu(s) \l F , \eta_{tx}(s)\r \d s
-\frac{\eps}{\sqrt\kappa} \int_0^\infty \mu(s) \big[( u,\eta_{t}(s))_1 + ( u_{t},\eta(s) )_1\big]\d s
\\&=  - \Gamma[\eta] + \frac{2}{\kappa} \int_0^\infty \mu(s) \l F , T \eta_{x}(s)\r \d s
-\frac{\eps}{\sqrt\kappa} \int_0^\infty \mu(s) ( u, T \eta(s))_1 \d s - \eps \sqrt{\kappa} \norm u \norm_1^2 \\
&\quad\, +\frac{\eps}{\sqrt\kappa} \int_0^\infty \mu(s) \l u_{x},\eta(s) \r \d s
+\frac{\eps}{\sqrt\kappa} \int_0^\infty \mu(s) \Big( \int_0^\infty \mu(\sigma)\l \eta(\sigma),\eta(s) \r_1 \d \sigma \Big) \d s\\
&\quad\, +\frac{\eps}{\sqrt\kappa} \int_0^\infty \mu(s) \l u u_{x},\eta(s) \r \d s
-\frac{\eps}{\sqrt\kappa} \int_0^\infty \mu(s) \l f ,\eta(s) \r \d s.
\end{align*}
Integrating by parts in $s$ (as shown in \cite{Terreni} the boundary terms vanish)
\begin{align*}
&\frac{2}{\kappa} \int_0^\infty \mu(s) \l F , T \eta_{x}(s)\r \d s
-\frac{\eps}{\sqrt\kappa} \int_0^\infty \mu(s) ( u, T \eta(s))_1 \d s \\
\noalign{\vskip0.5mm}
&= \frac{2}{\kappa} \int_0^\infty \mu'(s) \l F , \eta_{x}(s)\r \d s
-\frac{\eps}{\sqrt\kappa} \int_0^\infty \mu'(s) ( u, \eta(s))_1 \d s\\\noalign{\vskip1.3mm}
&\leq c\big[\|F\| + \eps \norm u \norm_1\big]\sqrt{\Gamma[\eta]}\\\noalign{\vskip1.3mm}
&\leq \frac14 \Gamma[\eta] + c \|F\|^2 + c \eps^2 \norm u \norm_1^2.
\end{align*}
We also estimate
\begin{align*}
&\frac{\eps}{\sqrt\kappa} \int_0^\infty \mu(s) \l u_{x},\eta(s) \r \d s
+\frac{\eps}{\sqrt\kappa} \int_0^\infty \mu(s) \Big( \int_0^\infty \mu(\sigma)\l \eta(\sigma),\eta(s) \r_1 \d \sigma \Big) \d s \\\noalign{\vskip1.7mm}
&\leq  c \eps \big[\norm u \norm_1  \|\eta\|_\M + \|\eta\|_\M^2 \big]\\\noalign{\vskip0.5mm}
&\leq  \frac{\delta}{8}\|\eta\|_\M^2 + c \eps^2 \norm u \norm_1^2 + c\eps \|\eta\|_\M^2.
\end{align*}
Moreover, in the light of the embedding $\HH^1\subset L^\infty(\I)$,
\begin{align*}
\frac{\eps}{\sqrt\kappa} \int_0^\infty \mu(s) \l u u_{x},\eta(s) \r \d s &\leq c \eps \|u u_x\| \|\eta\|_\M\\
&\leq c\eps \|u\|_{L^\infty} \|u_x\|\|\eta\|_\M \\
\noalign{\vskip1mm}
&\leq c\eps \norm u \norm_1^2 \|\eta\|_\M\\
& \leq  \frac{\delta}{16}\|\eta\|_\M^2 + c \eps^2 \E^2.
\end{align*}
Finally, integrating by parts, the remaining term is controlled by
\begin{align*}
-\frac{\eps}{\sqrt\kappa} \int_0^\infty \mu(s) \l f ,\eta(s) \r \d s &=
\frac{\eps}{\sqrt\kappa} \int_0^\infty \mu(s) \l F ,\eta_x(s) \r \d s\\\noalign{\vskip1.4mm}
&\leq c \eps \|F\|\|\eta\|_\M\\
& \leq \frac{\delta}{16}\|\eta\|_\M^2 + c \eps^2 \|F\|^2.
\end{align*}
Collecting all the estimates above, we get
$$
\ddt\L_\eps +  \frac{3}{4} \Gamma[\eta] + \eps (\sqrt{\kappa} - c\eps)\norm u\norm_1^2
\leq \Big(\frac{\delta}{4} + c\eps\Big)\|\eta\|_\M^2 + c(1+\eps^2) \|F\|^2 + c \eps^2 \E^2.
$$
At this point, owing to \eqref{gamm} and \eqref{finalcont}, we end up with the inequality
$$
\ddt\L_\eps +  \frac{\delta}{4} \|\eta\|_\M^2 + \frac{\eps \sqrt{\kappa}}{2} \norm u\norm_1^2
\leq  c \|F\|^2+ c \eps^2 \L_\eps^2,
$$
valid for all $\eps>0$ small enough. A final exploitation of \eqref{finalcont} completes the argument.
\end{proof}
%%%%%%%%%%%%%%%%%%%%%%%%%%%%%%%%%%%%%%%%%%%%%%%%%

%%%%%%%%%%%%%%%%%%%%%%%%%%%%%%%%%%%%%%%%%%%%%%%%%
\section{A Family of Invariant Sets}
\label{S8}

\noindent
The main assumption in this work is the following bound:
\begin{equation}
\label{small}
\|F\|< {\mathfrak c}\doteq\frac{c_1}{2 \sqrt{c_2c_3}},
\end{equation}
where $c_1,c_2,c_3$ are the constants appearing in Lemma~\ref{fi}.

\begin{remark}
It is worth pointing out that
$c_1,c_2,c_3$, which are independent of $f$, can be explicitly calculated
in such a way to maximize the value $\mathfrak c$.
\end{remark}

In what follows, we will always assume \eqref{small}. Then,
defining the number
$$
c_* =  \sqrt{\frac{c_2}{c_3}}\bigg(\frac{c_1}{\sqrt{c_2c_3}} -\|F\|\bigg)>0,
$$
we introduce the family of closed sets depending on $\eps \in (0,\eps_0]$, where $\eps_0$ comes from Lemma \ref{fi},
$$\DD_\eps= \Big\{ z \in \H : \Lambda_\eps (z)\leq \frac{c_*}{\eps}\Big\}.$$
In particular, $\DD_\eps$ turns out to be a complete metric space in the metric inherited
by $\H$.
In the next two lemmas, we collect some properties of the family $\DD_\eps$ needed in the sequel.

\begin{lemma}
\label{inin}
The set $\DD_\eps$ is bounded in $\H$
with\footnote{As usual, given a set $\B\subset\H$, we denote
$\|\B\|_\H = \sup_{z\in\H}\|z\|_\H.$}
$$\|\DD_\eps\|_\H \leq \frac{4c_*}{\eps}.$$
Moreover, for every bounded set $\B\subset\H$, we have the inclusion
$$
\B \subset \DD_\eps
$$
for all $\eps>0$ sufficiently small (depending on the $\H$-norm of $\B$).
\end{lemma}

\begin{proof}
Since $\eps\leq\eps_0<\tfrac{1}{2\omega}$, the $\H$-bound of $\DD_\eps$ is an immediate consequence
of the first inequality in \eqref{geoen}.
Moreover, given a bounded set $\B\subset\H$, the second inequality in \eqref{geoen} tells that
$$
\Lambda_\eps(z) \leq  2\|\B\|_{\H}^2 + \frac4\kappa\|F\|^2,\quad \forall z \in \B.
$$
Choosing
$$\eps \leq \frac{\kappa c_*}{2\kappa\|\B\|_{\H}^2+ 4\|F\|^2},$$
the right-hand side becomes less than or equal to $\tfrac{c_*}{\eps}$.
\end{proof}

\begin{lemma}
\label{inclusione}
Let $\eps \in (0,\eps_0]$ be fixed. Then,
for every $\alpha\leq \eps(1-\omega\eps)$, we have the inclusion
$$\DD_\eps\subset \DD_\alpha.$$
\end{lemma}

\begin{proof}
Let $z\in \DD_\eps$ be arbitrarily chosen. Owing to the first inequality in \eqref{calculus1},
$$
\Lambda_\alpha(z) \leq \frac{\Lambda_\eps(z)}{1-\omega\eps} \leq \frac{c_*}{\eps(1-\omega\eps)}.
$$
Being $\alpha\leq \eps(1-\omega\eps)$, we are finished.
\end{proof}

The forthcoming result will be of some importance.

\begin{proposition}
\label{invariance}
Assume that \eqref{small} holds. Then the set
$\DD_\eps$ is invariant for $S(t)$.
\end{proposition}

\begin{proof}
We need to prove that, for every $z\in\DD_\eps$ and every $t\geq0$,
$$
\L_\eps(t) = \Lambda_\eps(S(t) z) \leq \frac{c_*}{\eps}.
$$
Exploiting Lemma \ref{fi},
the functional $\L_\eps(t)$ satisfies the differential inequality
$$
\ddt \L_\eps(t)  + \eps c_1 \L_\eps(t) \leq c_2\|F\|^2 + c_3 \eps^2 \L_\eps^2(t),
$$
which is nothing but \eqref{DI} with
$$
\L = \L_\eps, \qquad a = c_3 \eps^2, \qquad b = \frac{\eps c_1}{2},\qquad c= c_2\|F\|^2.
$$
Accordingly, the constant $\varrho=\tfrac{b}{\sqrt{ac}}$ now reads
$$
\varrho  =\frac{c_1}{2 \sqrt{c_2c_3}\|F\|}>1,
$$
and it is independent of $\eps$.
Moreover,
$$
c_* = \sqrt{\frac{c_2}{c_3}} (2\varrho-1)\|F\|,
$$
and
$$
\lambda_{\pm}= \frac{c_*}{\eps} \,\frac{\varrho\pm\sqrt{\varrho^2-1}}{2\varrho-1}.
$$
It is then apparent that
$$
\lambda_- < \frac{c_*}{\eps} <\lambda_+,
$$
and by Lemma \ref{lemmaODE1} with $\lambda=\frac{c_*}{\eps}$ we are done.
\end{proof}
%%%%%%%%%%%%%%%%%%%%%%%%%%%%%%%%%%%%%%%%%%%%%%%%%

%%%%%%%%%%%%%%%%%%%%%%%%%%%%%%%%%%%%%%%%%%%%%%%%%
\section{A Technical Lemma}
\label{S9}

\noindent
For the proof of the main theorem a crucial inequality is needed,
involving the vectors lying simultaneously in a bounded subset of $\H$ and in the complement of a certain
$\DD_\eps$.

\begin{lemma}
\label{Oliv}
Assume that \eqref{small} holds. Then there exists $\eps_*\in(0,\eps_0]$ with the following property:\ for every bounded set
$\B \subset \H$ there is a time $\textsc{T}=\textsc{T}(\B)>0$ such that the inequality
$$
\Lambda_{\eps_*}(S(t)z) < \Lambda_{\eps_*}(z)
$$
holds
for all $z\in \B \cap \DD_{\eps_*}^{\rm c}$
and all $t\geq \textsc{T}$.
\end{lemma}

\begin{proof}
Analogously to the proof of Proposition \ref{invariance}, we set
$$
\varrho  =\frac{c_1}{2 \sqrt{c_2c_3}\|F\|}>1,
$$
where $c_1,c_2,c_3$ are the constants of Lemma \ref{fi}.
We divide the argument into four steps.

\medskip
\noindent
{\bf Step 0.} We begin to fix $\eps_*$ to be an arbitrarily given number subject to the constraint
$$
\eps_* \leq \min \Big\{ \eps_0,
\frac{1}{\omega}\left(\frac{\varrho-1}{3\varrho-2}\right)\Big\}.
$$
In particular, it is readily seen that the following inequalities hold:
\begin{align}
\label{prima}
&\frac{1}{1-\omega\eps_*}\leq 2\varrho-1, \\\noalign{\vskip0.7mm}
\label{seconda}
&\left(\frac{1-\omega\eps_*}{1-2\omega\eps_*}\right) \left(\frac{\varrho}{2\varrho-1}\right)\leq 1.
\end{align}
Next, for  $j\in\N=\{1,2,3,\ldots\}$, we introduce the sets
$$
\mathbb{H}_j=\left\{ z\in \H :\, \sqrt{\frac{c_2}{c_3}} \frac{\|F\|}{\eps_*}
\varrho^j< \Lambda_{\eps_*}(z)\leq \sqrt{\frac{c_2}{c_3}} \frac{\|F\|}{\eps_*} \varrho^{j+1}\right\}.
$$

\smallskip
\noindent
{\bf Step 1.} Our first aim is showing that there exists a number $n = n(\B)\in\N$ with the following property:\
for any given $z\in \B \cap \DD_{\eps_*}^{\rm c}$, there is $m=m(z)\leq n$ such that $z \in \mathbb{H}_m$.
This amounts to proving that
$$\B\cap \mathbb{D}_{\eps_*}^{\rm c} \subset \bigcup_{j=1}^{n} \mathbb{H}_j.
$$
Indeed, if $z\in \DD_{\eps_*}^{\rm c}$, then
$$
\Lambda_{\eps_*} (z) > \frac{c_*}{\eps_*}
=\sqrt{\frac{c_2}{c_3}} \frac{\|F\|}{\eps_*} (2\varrho-1)>\sqrt{\frac{c_2}{c_3}}\frac{\|F\|}{\eps_*} \varrho.
$$
In particular,
$$
\mathbb{D}_{\eps_*}^{\rm c} \subset \bigcup_{j=1}^{\infty} \mathbb{H}_j
=\left\{ z\in \H :\, \Lambda_{\eps_*}(z) > \sqrt{\frac{c_2}{c_3}} \frac{\|F\|}{\eps_*} \varrho \right\}.
$$
On the other hand, making use of the second inequality in \eqref{geoen}, for every $z \in \B$ we have
$$
\Lambda_{\eps_*}(z) \leq \frac{(3+\eps_*\omega)}{2}\|\B\|_{\H}^2 + \frac4\kappa\|F\|^2.
$$
Hence, choosing $n = n(\B)\in\N$ large enough that
$$
\sqrt{\frac{c_2}{c_3}} \frac{\|F\|}{\eps_*} \varrho^{n+1} \geq \frac{(3+\eps_*\omega)}{2}\|\B\|_{\H}^2 + \frac4\kappa\|F\|^2,
$$
we are led to
$$
\B \subset \left\{ z\in \H :\, \Lambda_{\eps_*}(z) \leq \sqrt{\frac{c_2}{c_3}} \frac{\|F\|}{\eps_*} \varrho^{n+1} \right\}.
$$
The claim is proved.

\medskip
\noindent
{\bf Step 2.}
Let now $z\in \B \cap \DD_{\eps_*}^{\rm c}$ be arbitrarily fixed,
and let $m=m(z)\leq n$ be the number constructed in Step 1, that is, $z\in\mathbb{H}_m$. We show that
the inequality
$$
\Lambda_{\eps_*}(S(t)z) < \Lambda_{\eps_*}(z)
$$
holds for every
$$
t \geq \frac{t_*}{\eps_*}\, \varrho^{m+1}
$$
where $t_*>0$ is independent of $z$.
To this end, introducing the number $\alpha=\alpha(z)<\eps_*$ as
$$\alpha=\frac{\eps_*}{\varrho^{m+1}}$$
and appealing to the first inequality in \eqref{calculus1}, together with the fact that $z \in \mathbb{H}_m$,
\begin{align*}
\Lambda_\alpha(z)\leq  \frac{\Lambda_{\eps_*}(z)}{1-\omega\eps_*}
\leq \sqrt{\frac{c_2}{c_3}}\frac{\|F\|}{(1-\omega\eps_*)}\frac{\varrho^{m+1}}{\eps_*}
= \sqrt{\frac{c_2}{c_3}}\frac{\|F\|}{(1-\omega\eps_*)}\frac{1}{\alpha}.
\end{align*}
Thus, owing to \eqref{prima},
we arrive at
\begin{equation}
\label{iniz}
\Lambda_\alpha(z)\leq \sqrt{\frac{c_2}{c_3}} \frac{\|F\|}{\alpha}( 2\varrho-1).
\end{equation}
In the light of Lemma \ref{fi},
the functional $\L_\alpha(t)=\Lambda_\alpha(S(t)z)$ satisfies the differential inequality
$$
\ddt \L_\alpha(t)  + \alpha c_1 \L_\alpha(t) \leq c_2\|F\|^2 + c_3 \alpha^2 \L_\alpha^2(t),
$$
which is nothing but \eqref{DI} with
$$
\L = \L_\alpha, \qquad a = c_3 \alpha^2, \qquad b = \frac{\alpha c_1}{2},\qquad c= c_2\|F\|^2.
$$
Due to \eqref{iniz}, we are in a position
to apply Lemma \ref{lemmaODE2}, obtaining
\begin{equation}
\label{gob}
\Lambda_\alpha(S(t)z)\leq \sqrt{\frac{c_2}{c_3}} \frac{\|F\|}{\alpha} \Big(\frac{1}{2\varrho-1}\Big)
\end{equation}
for every
$$t\geq  \frac{t_\varrho}{\alpha\sqrt{c_2c_3}\|F\|} = \frac{t_*}{\eps_*}\, \varrho^{m+1},$$
having set
$$
t_* = \frac{t_\varrho }{\sqrt{c_2c_3}\|F\|}>0.
$$
Being $\varrho$ independent of $z$, such a $t_*$ is independent of $z$ as well.
At this point, exploiting the estimate \eqref{gob} above and the second inequality in~\eqref{calculus1},
\begin{align*}
\Lambda_{\eps_*}(S(t)z)
&\leq \bigg(\frac{1-\omega\eps_*}{1-2\omega\eps_*}\bigg)\Lambda_\alpha(S(t)z)\\
&\leq \sqrt{\frac{c_2}{c_3}} \frac{\|F\|}{\alpha}  \bigg(\frac{1-\omega\eps_*}{1-2\omega\eps_*}\bigg)\bigg(\frac{1}{2\varrho-1}\bigg)\\
&= \sqrt{\frac{c_2}{c_3}} \frac{\|F\|}{\eps_*}\bigg(\frac{1-\omega\eps_*}{1-2\omega\eps_*}\bigg) \bigg(\frac{\varrho^{m+1}}{2\varrho-1}\bigg).
\end{align*}
Thanks to \eqref{seconda}, the last term is controlled by
$$
\sqrt{\frac{c_2}{c_3}} \frac{\|F\|}{\eps_*}\bigg(\frac{1-\omega\eps_*}{1-2\omega\eps_*}\bigg)
\bigg(\frac{\varrho^{m+1}}{2\varrho-1}\bigg)\leq
\sqrt{\frac{c_2}{c_3}} \frac{\|F\|}{\eps_*} \varrho^m.
$$
Summarizing,
$$
\Lambda_{\eps_*}(S(t)z) \leq \sqrt{\frac{c_2}{c_3}} \frac{\|F\|}{\eps_*} \varrho^m <\Lambda_{\eps_*}(z),
$$
where the latter inequality follows from $z\in \mathbb{H}_m$.

\medskip
\noindent
{\bf Step 3.} The sought \textsc{T} is defined as
$$
\textsc{T}=\textsc{T}(\B) =
\frac{t_* }{\eps_*} \varrho^{n+1},
$$
being $n\in\N$ the number constructed in Step 1 and $t_*>0$ the time constructed in Step 2.
Since
$$ \textsc{T}\geq \frac{t_* }{\eps_*} \varrho^{m+1},\quad  1\leq m\leq n,$$
we conclude that
$$
\Lambda_{\eps_*}(S(t)z) < \Lambda_{\eps_*}(z)
$$
for every $z\in \B \cap \DD_{\eps_*}^{\rm c}$
and every $t\geq \textsc{T}$.
\end{proof}
%%%%%%%%%%%%%%%%%%%%%%%%%%%%%%%%%%%%%%%%%%%%%%%%%

%%%%%%%%%%%%%%%%%%%%%%%%%%%%%%%%%%%%%%%%%%%%%%%%%
\section{A Family of Attractors}
\label{S10}

\noindent
For $\eps_0$ given by Lemma \ref{fi}, let now
$$\eps\in(0,\eps_0]$$
be arbitrarily fixed. We consider the restriction of $S(t)$ on the invariant complete metric space $\DD_\eps$.

\begin{theorem}
\label{attreps}
Assume that \eqref{small} holds. Then the semigroup $S(t): \DD_\eps\to\DD_\eps$ possesses
the global attractor $\A_\eps\subset\DD_\eps$.
Moreover, there exists a set $\mathcal{K}_\eps$ compact in $\H$ and bounded in $\H^1$ such that
$\A_\eps \subset \mathcal{K}_\eps$.
\end{theorem}

The remaining of the section is devoted to the proof of Theorem \ref{attreps}.

\subsection{The decomposition}
\label{deccom}
For an arbitrarily given initial datum $z\in\DD_\eps$,
we split the solution
$$
S(t)z = (u(t),\eta^t)
$$
into the sum
$$
(u(t),\eta^t) = (v(t),\xi^t) + (w(t),\zeta^t),
$$
where $(v(t),\xi^t)$ and $(w(t),\zeta^t)$ solve the Cauchy problems
\begin{equation}
\label{D1}
\begin{cases}
\displaystyle
B v_t + \int_0^\infty \mu(s) A \xi(s) \d s =0,\\\noalign{\vskip0.5mm}
\xi_t = T\xi + v,\\\noalign{\vskip2mm}
(v(0),\xi^0)=z,
\end{cases}
\end{equation}
and
\begin{equation}
\label{D2}
\begin{cases}
\displaystyle
B w_t  + \int_0^\infty \mu(s) A \zeta(s) \d s =f-u_x - u u_x,\\\noalign{\vskip0.5mm}
\zeta_t = T\zeta + w,\\\noalign{\vskip2mm}
(w(0),\zeta^0)=0,
\end{cases}
\end{equation}
respectively. Observe that, in general, neither $(v(t),\xi^t)$ nor $(w(t),\zeta^t)$ belong to $\DD_\eps$.
In what follows, $C=C(\DD_\eps)>0$ will denote a {\it generic} constant depending on $\DD_\eps$
and the structural quantities of the problem (including the external force $f$), but independent of the initial datum $z$.
In particular, the invariance of $\DD_\eps$ ensures that
\begin{equation}
\label{unif}
\norm u(t) \norm_1 + \|\eta^t\|_\M \leq C, \quad \forall t\geq0.
\end{equation}
The first step is proving the (exponential) decay of the solutions to \eqref{D1}.

\begin{lemma}
\label{exddec}
There exists a universal constant $\beta>0$ such that
$$
\|(v(t),\xi^t)\|_\H \leq C \e^{-\beta t}.
$$
\end{lemma}

In fact, although not needed in this context, $\beta$ turns out to be independent of $\DD_\eps$.
Actually, Lemma \ref{exddec} is just a byproduct
of~\cite{DMP}, where the exponential stability
of a more general (nonlinear) system has been proved. It is also worth mentioning that
the exponential stability of a closely related model, i.e.\ the Gurtin-Pipkin equation,
has been proved in \cite{GNP} by means of linear semigroup techniques.
Nevertheless, for the reader's convenience,
and in order to make the paper
self-contained, we report here a short proof based on explicit energy-type estimates.

\begin{proof}[Proof of Lemma \ref{exddec}]
We multiply the first equation of \eqref{D1} by $2v$ in $\HH$ and the second one
by $2\xi$ in $\M$. Taking the sum and invoking \eqref{FUNCGAMMA}, we obtain the identity
$$
\ddt \big[\norm v\norm^2_1 + \|\xi\|_\M^2 \big] +\Gamma[\xi] =0.
$$
We also consider the auxiliary functional
$$
\Phi(t) = -  \int_0^\infty \mu(s) ( v(t) ,  \xi^t(s))_1 \d s,
$$
which satisfies
\begin{align*}
\ddt \Phi &= -  \int_0^\infty \mu(s) ( v_t ,  \xi(s))_1 \d s -  \int_0^\infty \mu(s) ( v ,  \xi_t(s))_1 \d s\\
& =  \int_0^\infty \mu(s) \Big(\int_0^\infty \mu(\sigma) \l \xi(\sigma), \xi (s) \r_1 \d \sigma \Big) \d s
- \int_0^\infty \mu(s) (v,T\xi(s))_1\d s - \kappa\norm v \norm_1^2.
\end{align*}
It is clear that
$$
\int_0^\infty \mu(s) \Big(\int_0^\infty \mu(\sigma) \l \xi(\sigma), \xi (s) \r_1 \d \sigma \Big) \d s
\leq  \kappa\|\xi\|_\M^2.
$$
Moreover, integrating by parts in $s$,
$$
- \int_0^\infty \mu(s) (v,T\xi(s))_1\d s = -  \int_0^\infty \mu'(s) (v,\xi(s))_1\d s\\
\leq \sqrt{\mu(0)}\omega \norm v \norm_1 \sqrt{\Gamma[\xi]}.
$$
Collecting the calculations above and exploiting \eqref{gamm}
we easily see that, for every $\nu>0$, the functional
$$
\Theta_\nu(t) = \norm v(t)\norm^2_1 + \|\xi^t\|_\M^2 + \nu \Phi(t)
$$
fulfills the differential inequality
$$
\ddt \Theta_\nu + \nu \kappa \norm v \norm_1^2 + \frac{\delta}{2} \|\xi\|_\M^2+ \frac12 \Gamma[\xi]
\leq \frac{\nu^2 \mu(0)\omega^2}{2} \norm v \norm_1^2+\nu\kappa\|\xi\|_\M^2 + \frac12 \Gamma[\xi] .
$$
Possibly reducing $\nu>0$, the right-hand side is controlled by
$$
\frac{\nu^2 \mu(0)\omega^2}{2} \norm v \norm_1^2+\nu\kappa\|\xi\|_\M^2 + \frac12 \Gamma[\xi] \leq
\frac{\nu \kappa}{2} \norm v \norm_1^2+
\frac{\delta}{4} \|\xi\|_\M^2+ \frac12 \Gamma[\xi],
$$
yielding
$$
\ddt \Theta_\nu + \frac{\nu \kappa}{2} \norm v \norm_1^2 + \frac{\delta}{4} \|\xi\|_\M^2\leq0.
$$
It is also apparent to see that, for $\nu>0$ small enough,
$$
\frac12 \big[ \norm v(t)\norm^2_1 + \|\xi^t\|_\M^2 \big] \leq \Theta_\nu(t) \leq 2 \big[ \norm v(t)\norm^2_1 + \|\xi^t\|_\M^2 \big].
$$
Hence, there exists $\beta>0$ such that
$$
\ddt \Theta_\nu + 2 \beta \Theta_\nu \leq0.
$$
Finally, in the light of the Gronwall lemma,
$$
\norm v(t)\norm^2_1 + \|\xi^t\|_\M^2 \leq 2 \Theta_\nu(0) \e^{-2\beta t} \leq 4 \|z\|_\H^2
\e^{-2\beta t} \leq C \e^{-2\beta t},
$$
where the latter inequality follows from the boundedness of $\DD_\eps$.
\end{proof}

The next step is showing that the solutions to \eqref{D2} are uniformly bounded in $\H^1$.

\begin{lemma}
\label{regul}
There exists a structural constant $Q=Q(\DD_\eps)>0$ such that
$$
\sup_{t\geq0}\|(w(t),\zeta^t)\|_{\H^1} \leq Q.
$$
\end{lemma}

\begin{proof}
We preliminary observe that, due to the embedding $\HH^1\subset L^\infty(\mathfrak{I})$ and
the uniform estimate \eqref{unif},
\begin{equation}
\label{nonlin}
\| f-u_x - u u_x\| \leq \|f\| + \|u_x\| + \|u\|_{L^\infty}\|u_x\|
\leq C\big[1 + \norm u \norm_1 +   \norm u \norm_1^2 \big]
\leq C.
\end{equation}
Then, we multiply the first equation of \eqref{D2} by $2A w$ in $\HH$ and the second one
by $2A\zeta$ in $\M$. Taking the sum and exploiting \eqref{FUNCGAMMA}, we obtain
\begin{align}
\label{reg1}
\ddt \big[\norm w\norm^2_2 + \|\zeta\|_{\M^1}^2 \big] +\Gamma[A^{\frac12} \zeta] &=2 \l f- u_x- u u_x, A w \r\\\nonumber
&\leq 2 \| f-u_x - u u_x\| \|w\|_2\\\noalign{\vskip1.3mm}
&\leq C\norm w\norm_2,\nonumber
\end{align}
where the last inequality follows from \eqref{nonlin}.
Next, we introduce the auxiliary functional
$$
\Psi(t) = -  \int_0^\infty \mu(s) ( w(t) ,  \zeta^t(s))_2 \d s.
$$
Computing the time derivative,
\begin{align*}
\ddt \Psi &= -  \int_0^\infty \mu(s) ( w_t ,  \zeta(s))_2 \d s -  \int_0^\infty \mu(s) ( w ,  \zeta_t(s))_2 \d s\\
&=\int_0^\infty \mu(s) \Big(\int_0^\infty \mu(\sigma) \l \zeta(s),\zeta(\sigma)\r_2 \d \sigma \Big) \d s
-  \int_0^\infty \mu(s) \l f-u_x-u u_x , A \zeta(s)\r \d s\\
&\quad\,- \int_0^\infty \mu(s) (w,T\zeta(s))_2\d s - \kappa\norm w \norm_2^2.
\end{align*}
It is readily seen that
$$
\int_0^\infty \mu(s) \Big(\int_0^\infty \mu(\sigma) \l \zeta(s),\zeta(\sigma)\r_2 \d \sigma \Big) \d s
\leq \kappa \|\zeta\|_{\M^1}^2.
$$
In addition, invoking once more \eqref{nonlin},
\begin{align*}
-  \int_0^\infty \mu(s) \l f-u_x-u u_x , A \zeta(s)\r \d s
&\leq \sqrt{\kappa} \| f-u_x - u u_x\|  \|\zeta\|_{\M^1} \\
&\leq C \|\zeta\|_{\M^1}^2 +C.
\end{align*}
Finally, integrating by parts in $s$,
\begin{align*}
- \int_0^\infty \mu(s) (w,T\zeta(s))_2\d s  &= - \int_0^\infty \mu'(s) (w,\zeta(s))_2\d s \\
&\leq
 C \norm w \norm_2 \sqrt{\Gamma[A^{\frac12} \zeta]}\\\noalign{\vskip2mm}
 &\leq \frac{\kappa}{2}\norm w \norm_2^2 + C\Gamma[A^{\frac12} \zeta].
\end{align*}
In summary, the functional $\Psi$ fulfills
\begin{equation}
\label{psi}
\ddt \Psi +\frac{\kappa}{2}\norm w \norm_2^2
\leq  C\|\zeta\|_{\M^1}^2 +C\Gamma[A^{\frac12} \zeta] +C.
\end{equation}
At this point, for every $\nu>0$, we consider the further functional
$$
\Upsilon_\nu(t) = \norm w(t)\norm^2_2 + \|\zeta^t\|_{\M^1}^2  + \nu \Psi(t).
$$
With the aid of \eqref{gamm}, from \eqref{reg1} and \eqref{psi} we infer that
$$
\ddt \Upsilon_\nu +\frac{\nu\kappa}{2}\norm w \norm_2^2+ \frac{\delta}{2}\|\zeta\|_{\M^1}^2
+ \frac{1}{2} \Gamma[A^{\frac12}\zeta]
\leq C \norm w \norm_2 +  \nu C\big[\|\zeta\|_{\M^1}^2 +\Gamma[A^{\frac12} \zeta] +1\big].
$$
Up to taking $\nu>0$ small enough, the right-hand side can be estimated as
$$
 C \norm w \norm_2 +  \nu C\big[\|\zeta\|_{\M^1}^2 +\Gamma[A^{\frac12} \zeta] +1\big]
\leq \frac{\nu\kappa}{4}\norm w \norm_2^2 + \frac{\delta}{4}\|\zeta\|_{\M^1}^2
+  \frac{1}{2} \Gamma[A^{\frac12}\zeta] +   \frac{C}{\nu}.
$$
Hence, we arrive at
$$
\ddt \Upsilon_\nu +\frac{\nu\kappa}{4}\norm w \norm_2^2+ \frac{\delta}{4}\|\zeta\|_{\M^1}^2
\leq \frac{C}{\nu}.
$$
Since for all $\nu>0$ sufficiently small we also have the controls
$$
\frac12 \big[ \norm w(t)\norm^2_2 + \|\zeta^t\|_{\M^1}^2 \big]
\leq \Upsilon_\nu(t) \leq 2 \big[ \norm w(t)\norm^2_2 + \|\zeta^t\|_{\M^1}^2 \big],
$$
the differential inequality above yields
$$
\ddt \Upsilon_\nu + \nu^2  \Upsilon_\nu \leq \frac{C}{\nu}.
$$
Being $\Upsilon_\nu(0)=0$,
an application of the Gronwall lemma completes the argument.
\end{proof}

Finally, we prove that the solutions to \eqref{D2} lie in a compact set.
Indeed, this does not follow directly from
Lemma \ref{regul},
since the embedding
$\H^1 \subset \H$
is not compact due to the memory component (see \cite{PZ} for a counterexample).

\begin{lemma}
\label{regulazz}
There exists compact set $\mathcal{K}_\eps\subset\H$, which is also bounded in
$\H^1$, such that
$$
\bigcup_{t\geq 0}(w(t),\zeta^t)\subset \mathcal{K}_\eps.
$$
\end{lemma}

\begin{proof}
It is well known (see e.g.~\cite{Terreni}) that the second component $\zeta^t$
of the solution to \eqref{D2} admits
the explicit representation
formula
$$
\zeta^t(s)=
\begin{cases}
\int_0^s w(t-y)\d y & 0<s\leq t,\\\noalign{\vskip1.5mm}
\int_0^t w(t-y)\d y & s>t.
\end{cases}
$$
Thus,
$$
\partial_s\zeta^t (s)=
\begin{cases}
w(t-s) & 0<s\leq t,\\
0 & s>t,
\end{cases}
$$
and from Lemma \ref{regul} we deduce the bounds
$$
\|\partial_s\zeta^t\|_{\M^1} \leq  Q\sqrt{\kappa},\qquad
\|\zeta^t(s)\|_1^2 \leq  h(s),
$$
where
$$h(s)=\frac{Q^2s^2}{\lambda_1}\qquad\text{fulfills}\qquad
\int_0^\infty h(s)\mu(s)\d s<\infty.$$
Again, we recall that the bounds above are all independent on the
particular choice of $z\in\DD_\eps$. Using once more Lemma \ref{regul}, we conclude that
$(w(t),\zeta^t)$ remains confined for all times in the (closed) set
$$
\mathcal{K}_\eps = \Big\{(w,\zeta)\in\H:\,  \|(w,\zeta)\|_{\H^1} + \|\partial_s\zeta\|_{\M^1} \leq Q(\sqrt{\kappa}+1) \,\,\,\,
\text{and}\,\,\,\,
\|\zeta(s)\|_1^2 \leq h(s)\Big\}.
$$
Such a ${\mathcal K}_\eps$ is compact in $\H$, by a direct application
of a general compactness result from \cite{PZ}
(see Lemma 5.5 therein).
\end{proof}

\subsection{Conclusion of the proof of Theorem  \ref{attreps}}
By means of Lemmas \ref{exddec} and \ref{regulazz}, we readily get
$$
\lim_{t\to\infty}\,\boldsymbol{\delta}(S(t)\DD_\eps,\mathcal{K}_\eps)= 0.
$$
This in turn implies
$$
\lim_{t\to\infty}\, \boldsymbol{\alpha}_{\H} (S(t)\DD_\eps)\to0,
$$
where $\boldsymbol{\alpha}_{\H}(\B)$ is the
{\it Kuratowski measure of noncompactness}
of a bounded set $\B\subset \H$, defined as
$$\boldsymbol{\alpha}_{\H}(\B)=\inf\big\{d: \text{$\B$ is covered by finitely many
balls of diameter less than $d$}\big\}.$$
On the other hand, since $S(t)\DD_\eps\subset\DD_\eps$, it is easily verified that
$$\boldsymbol{\alpha}_{\DD_\eps} (S(t)\DD_\eps) \leq \boldsymbol{\alpha}_{\H} (S(t)\DD_\eps),$$
$\boldsymbol{\alpha}_{\DD_\eps}$ being the
Kuratowski measure of noncompactness on the space
$\DD_\eps$.
Accordingly,
$$
\lim_{t\to\infty}\, \boldsymbol{\alpha}_{\DD_\eps} (S(t)\DD_\eps)\to0.
$$
This fact, by a classical result on the theory of dynamical systems (see e.g. \cite{HAL}),
yields the existence of the global attractor $\A_\eps$.
We are left to prove the inclusion $\A_\eps\subset{\mathcal K}_\eps$.
To this end, let $z\in\A_\eps$ be arbitrarily fixed, and let $t_n\to\infty$
be a given sequence of times. By the full invariance of $\A_\eps$, for every $n$
there is $z_n\in\A_\eps$ such that
$$z=S(t_n)z_n.$$
Exploiting the decomposition of Subsection \ref{deccom}, together with Lemmas \ref{exddec} and \ref{regulazz},
$$z=z^1_n+z^2_n,$$
with $z^1_n\to 0$ in $\H$ as $n\to\infty$ and $z^2_n\in\mathcal{K}_\eps$ for every $n$.
Appealing now to the compactness of $\mathcal{K}_\eps$, we draw the convergence (up to a subsequence)
$z^2_n\to\bar z$, for some $\bar z\in\mathcal{K}_\eps$.
This entails the equality $z=\bar z$.
\qed
%%%%%%%%%%%%%%%%%%%%%%%%%%%%%%%%%%%%%%%%%%%%%%%%%

%%%%%%%%%%%%%%%%%%%%%%%%%%%%%%%%%%%%%%%%%%%%%%%%%
\section{The Global Attractor:\ Proof of Theorem \ref{attractor}}
\label{S11}

\noindent
Let the constant $\mathfrak c$ of Theorem \ref{attractor} be given by \eqref {small}.
The key argument is the next lemma, establishing the equality
of the attractors $\A_\eps$ found in Theorem~\ref{attreps}.
In what follows, $\eps_*$ is given by Lemma \ref{Oliv}.

\begin{lemma}
\label{uuu}
For every $\eps < \eps_*$ sufficiently small, we have the equality
$$
\A_\eps = \A_{\eps_*}.
$$
\end{lemma}

\begin{proof}
Let $\eps \leq \eps_*(1-\omega\eps_*)$ be fixed (recall that $\omega\eps_*<1$).
To reach the desired conclusion, it is enough showing that
\begin{equation}
\label{keyinc}
\A_\eps \subset \DD_{\eps_*}.
\end{equation}
Indeed, if \eqref{keyinc} holds, then $\A_\eps$ turns out to be a (bounded) fully invariant subset of $\DD_{\eps_*}$, hence
contained in $\A_{\eps_*}$ which is by definition the largest fully invariant subset
of $\DD_{\eps_*}$.
Moreover by Lemma~\ref{inclusione}
$$\DD_{\eps_*} \subset \DD_{\eps}
\quad \Rightarrow \quad
\A_{\eps_*} \subset \DD_{\eps}.
$$
This means that $\A_{\eps_*}$ is a (bounded) fully invariant subset of $\DD_{\eps}$, hence
contained in $\A_{\eps}$.
Accordingly, suppose \eqref{keyinc} false. Then, by the very definition of $\DD_{\eps_*}$,
$$
K \doteq \sup_{z\in\A_\eps} \Lambda_{\eps_*}(z) > \frac{c_*}{\eps_*}.
$$
Exploiting Lemma \ref{Oliv}, there exists $\textsc{T}=\textsc{T}(\A_\eps)>0$ such that
$$
\Lambda_{\eps_*}(S(\textsc{T})z) < \Lambda_{\eps_*}(z) \leq K, \quad \forall z\in\A_\eps \cap \DD_{\eps_*}^{\rm c}.
$$
On the other hand, since $\DD_{\eps_*}$ is invariant for $S(t)$,
$$
\Lambda_{\eps_*}(S(\textsc{T})z) \leq \frac{c_*}{\eps_*} < K, \quad \forall z\in\A_\eps \cap \DD_{\eps_*}
$$
In summary,
$$
\Lambda_{\eps_*}(S(\textsc{T})z) < K, \quad \forall z \in \A_\eps .
$$
Since $\A_\eps$ is compact
and $\Lambda_{\eps_*}$ is continuous, there is $v\in\A_\eps$ for which
$$
\Lambda_{\eps_*}(v) = K.
$$
At the same time, the full invariance of $\A_\eps$ ensures that
$$
v = S(\textsc{T})w
$$
for some $w\in\A_\eps$. In conclusion,
$$
K = \Lambda_{\eps_*}(v) = \Lambda_{\eps_*}(S(\textsc{T})w) <K,
$$
leading to a contradiction.
\end{proof}

\begin{remark}
Although this is beyond our scopes, Lemma \ref{uuu} can be shown to hold for all $\eps<\eps_*$.
\end{remark}

Once Lemma \ref{uuu} is established, completing the proof of Theorem \ref{attractor} is almost straightforward.
We show that, in fact,
$$\A=\A_{\eps_*}$$
is the sought global attractor.
Being $\A_{\eps_*}$ compact and fully invariant, we just need to verify the attraction property.
To this end, let $\B\subset\H$ be a bounded set. On account of Lemma \ref{inin} and Lemma~\ref{keyinc},
there exists $\eps=\eps(\B)<\eps_*$ such that
$$
\B\subset \DD_\eps\and \A_\eps =\A_{\eps_*}.
$$
Therefore, as $\A_\eps$ is attracting on $\DD_\eps$,
$$
\lim_{t\to\infty} \boldsymbol{\delta}(S(t) \B, \A_{\eps_*}) =
\lim_{t\to\infty} \boldsymbol{\delta}(S(t) \B, \A_{\eps}) = 0.
$$
Finally, since Theorem \ref{attreps} for $\eps=\eps_*$ provides the inclusion
$$
\A_{\eps_*} \subset \mathcal{K}_{\eps_*},
$$
the claimed boundedness in $\H^1$ of the global attractor readily follows.
\qed
%%%%%%%%%%%%%%%%%%%%%%%%%%%%%%%%%%%%%%%%%%%%%%%%%

%%%%%%%%%%%%%%%%%%%%%%%%%%%%%%%%%%%%%%%%%%%%%%%%%

%%%%%%%%%%%%%%%%%%%%%%%%%%%%%%%%%%%%%%%%%%%%%%%%%
\end{document}